\documentclass[a4paper, 12pt]{article}
\pdfoutput=1

\usepackage[T1]{fontenc}
\usepackage[utf8]{inputenc}

\usepackage{amsfonts,amsmath,amssymb,mathtools,dsfont,microtype,xfrac} 
\usepackage[usenames, dvipsnames]{xcolor} 
\usepackage[noBBpl]{mathpazo} 

\DeclareFontFamily{U} {cmr}{}

\DeclareFontShape{U}{cmr}{m}{n}{
	<-6> cmr5
	<6-7> cmr6
	<7-8> cmr7
	<8-9> cmr8
	<9-10> cmr9
	<10-12> cmr10
	<12-> cmr12}{}

\DeclareSymbolFont{Xcmr} {U} {cmr}{m}{n}
\DeclareMathSymbol{\Delta}{\mathord}{Xcmr}{'001}
\DeclareMathSymbol{\Upsilon}{\mathord}{Xcmr}{'007}
\DeclareMathSymbol{\Omega}{\mathord}{Xcmr}{'012}

\usepackage[labelsep = period, labelfont = bf, justification = centering]{caption}
\usepackage{float,graphicx,subcaption}
\usepackage{booktabs,makecell}
\usepackage{enumitem,mdwlist}
\setlist[itemize]{topsep=0ex,itemsep=0ex,parsep=0.4ex}
\setlist[enumerate]{topsep=0ex,itemsep=0ex,parsep=0.4ex}

\usepackage[left = 2.5cm, right = 2.5cm, top = 2.5cm, bottom = 2.5cm, headsep = 12pt, headheight = 15pt]{geometry}
\usepackage{parskip}

\usepackage[hyphens]{url} 
\usepackage[linktoc = all, hidelinks, colorlinks, unicode=true, pagebackref=true]{hyperref} 
\usepackage[capitalise, compress, nameinlink, noabbrev]{cleveref} 
\hypersetup{linkcolor={blue!70!black}, citecolor={green!70!black}, urlcolor={blue!70!black}}

\usepackage{amsthm,thmtools,thm-restate}
\declaretheorem[name = Theorem, numberwithin = section, style = plain]{theorem}

\declaretheorem[name = Corollary, numberlike = theorem, style = plain]{corollary}
\declaretheorem[name = Conjecture, numberlike = theorem, style = plain]{conjecture}
\declaretheorem[name = Definition, numberlike = theorem, style = definition]{definition}
\declaretheorem[name = Lemma, numberlike = theorem, style = plain]{lemma}

\declaretheorem[name = Problem, numberlike = theorem, style = plain]{problem}
\crefname{claim}{Claim}{Claims}
\Crefname{claim}{Claim}{Claims}
\crefname{conjecture}{Conjecture}{Conjectures}
\Crefname{conjecture}{Conjecture}{Conjectures}
\newcommand{\defn}[1]{\textcolor{Maroon}{\emph{#1}}}

\DeclareFontFamily{U}{matha}{\hyphenchar\font45}
\DeclareFontShape{U}{matha}{m}{n}{
	<5> <6> <7> <8> <9> <10> gen * matha
	<10.95> matha10 <12> <14.4> <17.28> <20.74> <24.88> matha12
}{}
\DeclareSymbolFont{matha}{U}{matha}{m}{n}

\DeclareMathSymbol{\specialuparrow}{\mathrel}{matha}{"D2}
\DeclareMathSymbol{\specialrightarrow}{\mathrel}{matha}{"D1}

\renewcommand*{\backref}[1]{}
\renewcommand*{\backrefalt}[4]{
	\ifcase #1 Not cited.%
	\or $\specialuparrow$#2%
	\else $\specialuparrow$#2%
	\fi%
}

\mathtoolsset{centercolon}

\newcommand{\eps}{\ensuremath{\varepsilon}}

\let\originalleft\left
\let\originalright\right
\renewcommand{\left}{\mathopen{}\mathclose\bgroup\originalleft}
\renewcommand{\right}{\aftergroup\egroup\originalright}
\renewcommand{\l}{\left}
\renewcommand{\r}{\right}

\newcommand{\st}{\ifnum\currentgrouptype=16 \mathrel{}\middle|\mathrel{}\else\mathrel{|}\fi}

\DeclarePairedDelimiter{\abs}{\lvert}{\rvert}
\DeclarePairedDelimiter{\floor}{\lfloor}{\rfloor}

\DeclarePairedDelimiter{\set}{\{}{\}}

\newcommand{\subs}{\subseteq}

\newcommand{\sm}{\setminus}
\newcommand{\sdiff}{\bigtriangleup}

\newcommand{\DeclareMath}[2]{\newcommand{#1}{\mathnormal{#2}}}

\DeclareMath{\N}{\mathbb{N}}
\DeclareMath{\calO}{\mathcal{O}}

\DeclareMathOperator{\pr}{\mathbb{P}}
\DeclareMathOperator{\ex}{\mathbb{E}}

\renewcommand*{\ge}{\geqslant}
\renewcommand*{\geq}{\geqslant}
\renewcommand*{\le}{\leqslant}
\renewcommand*{\leq}{\leqslant}

\title{Flashes and rainbows in tournaments}
\author{Ant\'{o}nio Gir\~{a}o\footnotemark[2]
\qquad Freddie Illingworth\footnotemark[2]
\qquad Lukas Michel\footnotemark[2]\\
Michael Savery\footnotemark[2]~\footnotemark[3]
\qquad Alex Scott\footnotemark[2]}
\date{1 June 2023}

\begin{document}
\maketitle

\begin{abstract}
    \noindent Colour the edges of the complete graph with vertex set $\set{1, 2, \dotsc, n}$ with an arbitrary number of colours. What is the smallest integer $f(l,k)$ such that if $n > f(l,k)$ then there must exist a monotone monochromatic path of length $l$ or a monotone rainbow path of length $k$? Lefmann, R\"{o}dl, and Thomas conjectured in 1992 that  $f(l, k) = l^{k - 1}$ and proved this for $l \ge (3 k)^{2 k}$. We prove the conjecture for $l \geq k^3 (\log k)^{1 + o(1)}$ and establish the general upper bound $f(l, k) \leq k (\log k)^{1 + o(1)} \cdot l^{k - 1}$. This reduces the gap between the best lower and upper bounds from exponential to polynomial in $k$. We also generalise some of these results to the tournament setting.
\end{abstract}
  
\renewcommand{\thefootnote}{\fnsymbol{footnote}} 

\footnotetext[0]{\emph{2020 MSC}: 05C55 (generalised Ramsey theory), 05C35 (extremal problems), 05C38 (paths and cycles).}

\footnotetext[2]{Mathematical Institute, University of Oxford, United Kingdom (\textsf{\{\href{mailto:girao@maths.ox.ac.uk}{girao},\href{mailto:illingworth@maths.ox.ac.uk}{illingworth},\href{mailto:michel@maths.ox.ac.uk}{michel},\href{mailto:savery@maths.ox.ac.uk}{savery},\allowbreak \href{mailto:scott@maths.ox.ac.uk}{scott}\}@maths.ox.ac.uk}). Research of AG, FI, and AS supported by EPSRC grant EP/V007327/1.}
\footnotetext[3]{Heilbronn Institute for Mathematical Research, Bristol, UK.}

\renewcommand{\thefootnote}{\arabic{footnote}} 

\section{Introduction}

In 1930, Ramsey~\cite{R1930} showed that every $k$-colouring of the edges of a very large clique contains a large monochromatic sub-clique. This classical theorem was the starting seed for Ramsey theory, and there have been numerous extensions and generalisations of it since then. One of the most important is the canonical Ramsey theorem of Erd\H{o}s and Rado~\cite{ER1950}, which considers colourings that allow an arbitrary number of colours. It states that every colouring of the edges of a very large clique with an arbitrary number of colours contains a large sub-clique with one of four types of colourings.

\begin{theorem}[Erd\H{o}s, Rado]
    There is a function $f(k)$ such that every colouring of the edges of the complete graph on $\{1,2,\ldots, f(k)\}$ contains a complete $k$-vertex subgraph whose colouring is of one of the following four canonical types\textup{:}
    \begin{itemize}
        \item rainbow --- no two edges have the same colour;
        \item monochromatic --- all edges have the same colour;
        \item upper lexical --- two edges have the same colour if and only if they end at the same vertex;
        \item lower lexical --- two edges have the same colour if and only if they start at the same vertex.
    \end{itemize} 
\end{theorem}

Lefmann and R\"odl~\cite{LR1995} gave the best current bounds for this result, showing that we can take $f(k) = k^{Ck^2}$ for some constant $C$. The canonical Ramsey theorem has an immediate corollary for monotone paths. It implies that there is an integer $f(l, k)$ such that if $n > f(l, k)$ then every colouring $c \colon \tbinom{[n]}{2} \to \N$ yields either:
\begin{itemize}
    \item $x_0 < \dotsb < x_l$ with $c(x_0 x_1), c(x_1 x_2), \dotsc, c(x_{l - 1} x_l)$ all the same (an \defn{$l$-flash}), or
    \item $y_0 < \dotsb < y_k$ with $c(y_0 y_1), c(y_1 y_2), \dotsc, c(y_{k - 1} y_k)$ all distinct (a \defn{$k$-rainbow}).
\end{itemize}
Lefmann and R\"{o}dl's bound implies that $f(k,k) \leq k^{Ck^2}$.

In 1992, Lefmann, R\"{o}dl, and Thomas~\cite{LRT1992} gave better bounds on $f(l,k)$. From below, they provided a nice construction showing that $f(l, k) \geq l^{k - 1}$: label the vertices with strings from $\set{1, 2, \dotsc, l}^{k - 1}$ in lexicographic order and, for a pair of vertices $u < v$, let $c(u v)$ be any index at which $u$ has a smaller value than $v$. It is not difficult to see that this construction contains no $l$-flash or $k$-rainbow, and in fact it is best possible if only $k - 1$ colours are used. They made the attractive conjecture that this lower bound is tight in general.

\begin{conjecture}[Lefmann, R\"{o}dl, Thomas]\label{conj:fequality}
    For all positive integers $l$ and $k$, $f(l, k) = l^{k - 1}$.
\end{conjecture}

Lefmann, R\"{o}dl, and Thomas gave support for their conjecture by proving that it is true for the small cases $l \leq 2$ or $k \leq 4$ as well as when $l$ is at least factorial in $k$ (specifically, when $l \geq (3k)^{2k}$). Our first result proves \cref{conj:fequality} provided $l$ is at least polynomial in $k$ (specifically, when $l \geq k^3 (\log k)^{1 + o(1)}$).

\begin{theorem}\label{thm:fequality}
    For all $\eps > 0$, there is $C_{\eps}$ such that if $l \geq C_{\eps} k^3 (\log k)^{1 + \eps}$, then $f(l, k) = l^{k - 1}$.
\end{theorem}

The best general upper bound on $f(l,k)$ was given by Jiang and Mubayi~\cite{JM2000}.

\begin{theorem}[Jiang, Mubayi]\label{thm:fasymptotic}
    For all integers $l \geq 1$ and $k \geq 4$,
    \begin{equation*}
        f(l, k) \le \biggl(1 + \frac{1}{\sqrt{l}}\biggr)^{k - 4} \cdot l^{k - 1}.
    \end{equation*}
\end{theorem}

This confirms \cref{conj:fequality} asymptotically when $l = \omega(k^2)$. However, for $l$ fixed and $k$ large this differs from the lower bound by a factor exponential in $k$. We provide an upper bound in which this factor is only polynomial in $k$.

\begin{theorem}\label{thm:fupperbound}
    For all $\eps > 0$, there is $D_{\eps}$ such that for all positive integers $l$ and $k$,
    \begin{equation*}
        f(l, k) \leq D_{\eps} k (\log k)^{1 + \eps} \cdot l^{k - 1}.
    \end{equation*}
\end{theorem}

\Cref{thm:fequality,thm:fasymptotic,thm:fupperbound} are each best in different ranges of values of $l$ and $k$. For $l \geq k^3 (\log k)^{1 + o(1)}$, \cref{thm:fequality} is best possible, for $l = \calO((k/\log k)^2)$, \cref{thm:fupperbound} provides the best bound known, and \cref{thm:fasymptotic} is best in the intermediate regime.

Flashes and rainbows generalise very naturally to tournaments, where a \defn{tournament} is an orientation of a complete graph. A \defn{directed walk} in a tournament is a sequence of vertices $x_0, \dots, x_l$ such that $x_{i-1} x_i$ is a directed edge for all $1 \le i \le l$ (note that walks are allowed to repeat vertices and edges).

\begin{definition}
    Let $l$ and $k$ be non-negative integers. An \defn{$l$-flash} in a tournament is a directed walk $x_0, \dotsc, x_l$ with $c(x_0 x_1), \dotsc, c(x_{l - 1} x_l)$ all the same. A \defn{$k$-rainbow} is a directed walk $y_0, \dotsc, y_k$ with $c(y_0 y_1), \dotsc, c(y_{k - 1} y_k)$ all distinct.
\end{definition}

Let $t(l,k)$ denote the smallest integer such that every colouring of the edges of any tournament with more than $t(l,k)$ vertices contains an $l$-flash or a $k$-rainbow. \Cref{conj:fequality} can now be viewed as the restriction of this problem to \emph{transitive} tournaments and thus we have
\begin{equation*}
    t(l, k) \geq f(l, k) \geq l^{k - 1}.
\end{equation*}
We conjecture that the tournament strengthening of \cref{conj:fequality} holds.

\begin{conjecture}\label{conj:tequality}
    For all positive integers $l$ and $k$, $t(l, k) = l^{k - 1}$.
\end{conjecture}

We prove two upper bounds on $t(l, k)$. The first, \cref{thm:tasymptotic}, is better when $l$ is large compared to $k$ and confirms \cref{conj:tequality} asymptotically for $l = \omega(k^2)$. Although the bound looks very similar to \cref{thm:fasymptotic}, the proof is quite different and uses a probabilistic argument.

\begin{theorem}\label{thm:tasymptotic}
    For all positive integers $l$ and $k$, $t(l, k) \leq \bigl(1 + \frac{2}{\sqrt{l}}\bigr)^k \cdot l^{k-1}$.
\end{theorem}

Our second upper bound, \cref{thm:tupperbound}, is better when $k$ is large compared to $l$ and, in particular, improves upon \cref{thm:tasymptotic} for $l = \calO((k/\log k)^2)$. Note that it implies \cref{thm:fupperbound}.

\begin{theorem}\label{thm:tupperbound}
    For every $\eps > 0$, there is $D_{\eps}$ such that for all positive integers $l$ and $k$,
    \begin{equation*}
        t(l, k) \leq D_{\eps} k (\log k)^{1 + \eps} \cdot l^{k - 1}.
    \end{equation*}
\end{theorem}

\Cref{conj:tequality} posits that if one wants to colour the edges of a large tournament while avoiding $l$-flashes and $k$-rainbows, then transitive tournaments are best in the sense of having the most vertices. The following result shows that there are tournaments that are strictly worse than transitive tournaments.

\begin{theorem}\label{thm:tworse}
    For all $l \ge 2$ and $k \geq 3$, there is a tournament with $f(l,k) - 1$ vertices such that every colouring of its edges contains an $l$-flash or a $k$-rainbow.
\end{theorem}

Finally, it is interesting to ask how bad a tournament can be: how few vertices can a tournament have if every colouring of its edges contains an $l$-flash or a $k$-rainbow? The following provides a general lower bound showing that the number of vertices in such a tournament must be within a factor $l \sqrt{k}$ of $l^{k - 1}$.

\begin{theorem}\label{thm:tlowerbound}
    There is a constant $C > 0$ such that every tournament with at most $C l^{k - 2}/\sqrt{k}$ vertices has an edge colouring with neither $l$-flashes nor $k$-rainbows.
\end{theorem}

The rest of the paper is structured as follows. In \cref{sec:boundstournaments} we consider arbitrary tournaments and prove \cref{thm:tasymptotic,thm:tupperbound}. We then specialise to transitive tournaments in \cref{sec:boundstransitive} and provide a proof of \cref{thm:fequality}. In \cref{sec:colouringtournaments} we explore the differences between these two settings, proving \cref{thm:tworse,thm:tlowerbound}, and we close by stating some open problems in \cref{sec:openproblems}. Throughout this paper, $l$ and $k$ are positive integers.

\section{Upper bounds for arbitrary tournaments}
\label{sec:boundstournaments}

In this section, we prove \cref{thm:tasymptotic,thm:tupperbound}. For both, we rely on the following observation: if a tournament contains no $k$-rainbow and a $(k-1)$-rainbow ends at some vertex $v$, then $v$ has at most $k-1$ outgoing colours. Indeed, otherwise any $(k-1)$-rainbow ending at $v$ could be extended to a $k$-rainbow. This will allow us to restrict ourselves to tournaments where each vertex only has few outgoing (and incoming) colours. In this setting, we show that the following lemma holds.
  
\begin{lemma}\label{lem:cinoutcols}
    Let $T = (V, E)$ be an edge-coloured tournament containing no $l$-flash, and assume that every vertex has at most $c$ incoming colours and at most $c$ outgoing colours. Then,
    \[
        \abs{V} \le \l(1 + \frac{2}{\sqrt{l}}\r)^c \cdot l^c.
    \]
\end{lemma}

We prove this result with a probabilistic argument. We will use the following definitions. For a vertex $v$, we write $c^-(v) = \{ c(u v) \st u v \in E \}$ for the set of incoming colours of $v$, $c^+(v) = \{ c(v u) \st v u \in E \}$ for the set of outgoing colours of $v$, and $c(v) = c^-(v) \cup c^+(v)$ for the set of colours incident to $v$. Moreover, for every vertex $v$ and colour $a$, let
\[
    l_a(v) = \begin{cases}
        l-1 & \text{if } a \in c^-(v) \sm c^+(v), \\
        \text{length of the longest flash of colour $a$ ending at $v$} & \text{otherwise}.
    \end{cases}
\]
This is defined for every colour $a$ in our colouring; however $l_a(v)$ is only non-zero for incoming colours of $v$. Note also that if an edge $u v$ has colour $a$, then $l_a(u) < l_a(v)$.

\begin{proof}[Proof of \cref{lem:cinoutcols}]
    Let $p \ge 0$ be the root of $(l-2)p^2+2p=1$, so $p = 1 / (1 + \sqrt{l-1})$. For every colour $a$, choose $l_a \in \{0, \dots, l-1\}$ independently at random such that $\pr(l_a = 0) = \pr(l_a = l-1) = p$ and $\pr(l_a = i) = p^2$ for every $i \in [l-2]$. Define
    \[
        U = \{ v \in V \st \forall a \in c(v),\ l_a(v) = l_a \}.
    \]
    If $T[U]$ contains an edge $u v$, say of colour $a$, then $a \in c(u) \cap c(v)$ and $l_a(u) < l_a(v)$, contradicting the definition of $U$. Thus, $\abs{U} \le 1$. Moreover, for every vertex $v$ and every colour $a \in c(v)$ it holds that $l_a(v) \in \{0, l-1\}$ if and only if $a \in c^-(v) \sdiff c^+(v)$, and so
    \[
        \pr(v \in U) = p^{\abs{c^-(v) \sdiff c^+(v)}} p^{2 \abs{c^-(v) \cap c^+(v)}} = p^{\abs{c^-(v)} + \abs{c^+(v)}} \ge p^{2 c}.
    \]
    It follows that the expected size of $U$ is at least $p^{2 c} \abs{V}$, and since $\abs{U} \le 1$ we have
    \begin{align*}
        \abs{V} & \le p^{-2 c} = \l(1 + \sqrt{l-1}\r)^{2 c} = \l(l + 2 \sqrt{l-1}\r)^c \le \l(1 + \frac{2}{\sqrt{l}}\r)^c \cdot l^c. \qedhere
    \end{align*}
\end{proof}

To obtain \cref{thm:tasymptotic}, we will use induction to bound the number of vertices where no $(k-1)$-rainbow starts and no $(k-1)$-rainbow ends. By the observation from the beginning of this section, the number of remaining vertices can then be bounded via the preceding lemma.

\begin{proof}[Proof of \cref{thm:tasymptotic}]
    For $k \ge 2$, we claim that $t(l,k) \le 2 \cdot t(l,k-1) + (l + 2 \sqrt{l})^{k-1}$. Indeed, let $T = (V, E)$ be a tournament containing no $l$-flash and no $k$-rainbow. Define
    \begin{align*}
        & R = \{ v \in V \st \text{no $(k-1)$-rainbow ends at $v$} \} \text{ and} \\
        & S = \{ v \in V \st \text{no $(k-1)$-rainbow starts at $v$} \}.
    \end{align*}
    Clearly $\abs{R}, \abs{S} \le t(l,k-1)$. Every vertex $v \in V \sm R$ has at most $k-1$ outgoing colours (or else any $(k-1)$-rainbow ending at $v$ could be extended to a $k$-rainbow), and similarly every vertex in $V \sm S$ has at most $k-1$ incoming colours. By \cref{lem:cinoutcols}, it follows that $\abs{V \sm (R \cup S)} \le (1 + 2 / \sqrt{l})^{k-1} \cdot l^{k-1}$ and therefore
    \[
        \abs{V} \le \abs{R} + \abs{S} + \abs{V \sm (R \cup S)} \le 2 \cdot t(l,k-1) + (l + 2 \sqrt{l})^{k-1}.
    \]
    This proves the claim. A straightforward induction on $k$ with base case $t(l,1) = 1$ yields
    \begin{align*}
        t(l,k) & \le \sum_{i=0}^{k-1} \l(\frac{2}{l + 2 \sqrt{l}}\r)^i (l + 2 \sqrt{l})^{k-1} \le \frac{1}{1 - \frac{2}{l + 2 \sqrt{l}}} (l + 2 \sqrt{l})^{k-1} \\
        & = \frac{l + 2 \sqrt{l}}{l + 2 \sqrt{l} - 2} \l(1 + \frac{2}{\sqrt{l}}\r)^{k-1} \cdot l^{k-1} \le \l(1 + \frac{2}{\sqrt{l}}\r)^k \cdot l^{k-1}. \qedhere
    \end{align*}
\end{proof}

Next, we prove \cref{thm:tupperbound}. To this end, we will consider so-called `robust' vertices, which were introduced by Lefmann, R\"{o}dl, and Thomas \cite{LRT1992} under a different name. Loosely, we call a vertex $v$ \emph{robust} if we can find long rainbows ending and/or starting at $v$ that avoid any given colour.

\begin{definition}
    Let $r$ be a non-negative integer. A vertex $v$ is \defn{$r$-in-robust} if, for every colour $a$, some $r$-rainbow without colour $a$ ends at $v$.
\end{definition}

In-robust vertices are very useful for constructing long rainbows. Indeed, assume that $u$ is an in-robust vertex with an edge of colour $a$ to some other vertex $v$. Then, we can pick a long rainbow without colour $a$ ending at $u$ and try to connect this rainbow to some rainbow starting at $v$. Since the edge $u v$ cannot interfere with the rainbow ending at $u$, choosing an appropriate rainbow starting at $v$ might result in a long rainbow.

Of course, in-robust vertices are only useful if they are forced to occur in large enough tournaments. As above, it will suffice to consider tournaments with few incoming and outgoing colours at each vertex. Therefore, let $g(l,k,r)$ denote the smallest integer such that if $T$ is a tournament with more than $g(l,k,r)$ vertices, then every edge colouring of $T$ with at most $k-1$ outgoing colours at each vertex contains an $l$-flash, a $k$-rainbow, or an $r$-in-robust vertex.

It turns out that $g(l,k,r)$ is not much larger than $t(l,r)$. Therefore, in large tournaments we will be able to find $r$-in-robust vertices for large $r$, which we can then use to find $k$-rainbows as outlined above. This strategy yields the following bound on $t(l,k)$.

\begin{lemma}\label{lem:trecursiveupperbound}
    For every $k \ge 2$ and $1 \le r < k$,
    \[
        t(l,k) \le t(l,k-1) + g(l,k,r-1) + 1 + 2 (k-1) l^r \cdot t(l,k-r).
    \]
\end{lemma}

\begin{proof}
    Let $T = (V, E)$ be an edge-coloured tournament with no $l$-flash or $k$-rainbow. Let
    \[
        R = \{ v \in V \st \text{no $(k-1)$-rainbow ends at $v$} \}.
    \]
    Then $\abs{R} \le t(l,k-1)$, and every vertex in $V \sm R$ has at most $k-1$ outgoing colours.

    Order the vertices in $V \sm R$ in a way that maximises the number of forward edges of $T[V \sm R]$, and pick a minimal initial segment $P \subs V \sm R$ of this ordering such that $T[P]$ contains an $(r-1)$-in-robust vertex $v$ (if no such vertex exists, then we already have $\abs{V} \le \abs{R} + \abs{V \sm R} \le t(l,k-1) + g(l,k,r-1)$). Note that $\abs{P} \le g(l,k,r-1) + 1$. By our choice of ordering on $V\sm R$, $v$ must have outgoing edges to at least half of the vertices following it, so if $U = \{ u \in V \st v u \in E \}$ denotes the out-neighbourhood of $v$, then $\abs{V \sm (R \cup P)} \le 2 \abs{U}$.

    For every colour $a \in c^+(v)$, let $U_a = \{ u \in U \st c(v u) = a \}$. Since $v$ is $(r-1)$-in-robust, we can choose an $(r-1)$-rainbow $W_a$ without colour $a$ that ends at $v$, and extend it to each of the vertices in $U_a$ by the edges of colour $a$. Let $C_a$ be the set containing $a$ and the $r-1$ colours appearing in $W_a$.

    Next, for every $m \in \{0, \dots, l-1\}^{C_a}$ define
    \[
        U_a^m = \{ u \in U_a \st \forall b \in C_a,\ l_b(u) = m_b \}.
    \]
    If the colour of an edge $u w$ in $T[U_a^m]$ is $b$, then $l_b(u) < l_b(w)$, and so $b \notin C_a$. Thus, if $T[U_a^m]$ contained a $(k-r)$-rainbow starting at some vertex $u$, then this rainbow would use no colour from $C_a$ and could therefore be extended by the $r$-rainbow ending at $u$ to a $k$-rainbow. It follows that $T[U_a^m]$ contains no $(k-r)$-rainbow, so $\abs{U_a^m} \le t(l,k-r)$ and
    \[
        \abs{U_a} = \sum_{m \in \{0, \dots, l-1\}^{C_a}} \abs{U_a^m} \le \sum_{m \in \{0, \dots, l-1\}^{C_a}} t(l,k-r) \le l^r \cdot t(l,k-r).
    \]
    Since $v$ has at most $k-1$ outgoing colours, this implies that
    \[
        \abs{U} = \sum_{a \in c^+(v)} \abs{U_a} \le \sum_{a \in c^+(v)} l^r \cdot t(l,k-r) \le (k-1) l^r \cdot t(l,k-r).
    \]
    Putting everything together, we get
    \begin{align*}
        \abs{V} & = \abs{R} + \abs{P} + \abs{V \sm (R \cup P)} \\
        & \le \abs{R} + \abs{P} + 2 \abs{U} \\
        & \le t(l,k-1) + g(l,k,r-1) + 1 + 2 (k-1) l^r \cdot t(l,k-r). \qedhere
    \end{align*}
\end{proof}

To apply this lemma, we need a good bound on $g(l,k,r)$. Fortunately, we can use a strategy similar to that used the preceding proof to obtain the following result.

\begin{lemma}\label{lem:grecursiveupperbound}
    For every $r \ge 1$, it holds that
    \[
        g(l,k,r) \le g(l,k,r-1) + 1 + 2 (k-1) r l \cdot t(l,r).
    \]
\end{lemma}

\begin{proof}
    Let $T = (V, E)$ be an edge-coloured tournament with no $l$-flash, no $k$-rainbow, no $r$-in-robust vertex, and at most $k-1$ outgoing colours at each vertex. Order the vertices of $T$ in a way that maximises the number of forward edges, and pick a minimal initial segment $P \subs V$ of this ordering such that $T[P]$ contains an $(r-1)$-in-robust vertex $v$. As before, $\abs{P} \le g(l,k,r-1) + 1$ and if $U = \{ u \in V \st v u \in E \}$ denotes the out-neighbourhood of $v$, then $\abs{V \sm P} \le 2 \abs{U}$. For each $a \in c^+(v)$, define $U_a$, $W_a$, and $C_a$ as in the previous proof.

    Next, for every colour $b \in C_a$, define
    \[
        U_{a,b} = \{ u \in U_a \st \text{every $r$-rainbow ending at $u$ contains colour $b$} \}.
    \]
    By assumption, no vertex is $r$-in-robust, but for each $u\in U_a$ there is an $r$-rainbow using only the colours of $C_a$ ending at $u$. It follows that $U_a = \bigcup_{b \in C_a} U_{a,b}$. Finally, for every $m \in \{0, \dots, l-1\}$, let
    \[
        U_{a,b}^m = \{ u \in U_{a,b} \st l_b(u) = m \}.
    \]
    Now, if an edge $u w$ in $T[U_{a,b}^m]$ has colour $b$, then $l_b(u) < l_b(w)$ which gives a contradiction. Thus, if $T[U_{a,b}^m]$ contained an $r$-rainbow ending at some vertex $u$, then this rainbow would not use the colour $b$, contradicting $u \in U_{a,b}$. It follows that $T[U_{a,b}^m]$ contains no $r$-rainbow, so $\abs{U_{a,b}^m} \le t(l, r)$ and
    \[
        \abs{U_a} \le \sum_{b \in C_a} \abs{U_{a,b}} = \sum_{b \in C_a} \sum_{m \in \{0, \dots, l-1\}} \abs{U_{a,b}^m} \le \sum_{b \in C_a} \sum_{m \in \{0, \dots, l-1\}} t(l,r) = r l \cdot t(l,r).
    \]
    Since $v$ has at most $k-1$ outgoing colours, this implies $\abs{U} \le (k-1) r l \cdot t(l,r)$, and so
    \[
        \abs{V} = \abs{P} + \abs{V \sm P} \le \abs{P} + 2 \abs{U} \le g(l,k,r-1) + 1 + 2 (k-1) r l \cdot t(l,r). \qedhere
    \]
\end{proof}

Finally, to obtain numerical bounds on $t(l,k)$ and $g(l,k,r)$, we apply these recursive bounds repeatedly. This is handled by the following corollary, which proves \cref{thm:tupperbound} and gives an upper bound on $g(l,k,r)$ that will be useful in \cref{sec:boundstransitive}.

\begin{corollary}\label{cor:tgupperbounds}
    For every $\eps > 0$, there is $D_\eps$ such that for $r<k$
    \[
        t(l,k) \le D_\eps k (\log k)^{1+\eps} \cdot l^{k-1} \quad \text{and} \quad g(l,k,r) \le 4 D_\eps k^3 (\log k)^{1+\eps} \cdot l^r.
    \]
\end{corollary}

\begin{proof}
    From \cref{lem:trecursiveupperbound} and \cref{lem:grecursiveupperbound}, we can obtain the inequalities
    \begin{align*}
        t(l,k) & \le t(l,k-1) + g(l,k,r-1) + 2 k l^r \cdot t(l,k-r) \text{ and} \\
        g(l,k,r) & \le g(l,k,r-1) + 2 k^2 l \cdot t(l,r)
    \end{align*}
    for $k \ge 2$ and $1 \le r < k$.
    
    Let $k_0\geq 2$ be such that $8 (3 \log k + 5) (\log(3 \log k + 5))^{1+\eps} \le (\log k)^{1+\eps}$ for all $k \ge k_0$. By \cref{thm:tasymptotic}, there is $D_\eps$ such that $t(l,k) \le D_\eps k (\log k)^{1+\eps} \cdot l^{k-1}$ for all $k < k_0$.

    We prove the corollary by induction on $k$. Assume that we have already proved it for all $k' < k$. We may also assume that $l \ge 2$ since the claim is trivial for $l = 1$. Then, for $r < k$, a simple induction on $r$ with base case $g(l,k,0) = 0$ yields
    \begin{align*}
        g(l,k,r) & \le g(l,k,r-1) + 2 k^2 l \cdot t(l,r) \\
        & \le 4 D_\eps k^3 (\log k)^{1+\eps} \cdot l^{r-1} + 2 D_\eps k^3 (\log k)^{1+\eps} \cdot l^r \stackrel{l \ge 2}{\le} 4 D_\eps k^3 (\log k)^{1+\eps} \cdot l^r.
    \end{align*}
    To bound $t(l,k)$, we may assume that $k \ge k_0$ by our choice of $D_\eps$. Note that
    \[
        t(l,k-1) \le D_\eps k (\log k)^{1+\eps} \cdot l^{k-2} \stackrel{l \ge 2}{\le} \frac{D_\eps k (\log k)^{1+\eps} \cdot l^{k-1}}{2}.
    \]
    Choose
    \[
        r = \floor*{k - \frac{2 \log k + \log 16}{\log l}},
    \]
    then
    \[
        g(l,k,r-1) \le 4 D_\eps k^3 (\log k)^{1+\eps} \cdot l^{k - 1 - \frac{2 \log k + \log 16}{\log l}} = \frac{D_\eps k (\log k)^{1+\eps} \cdot l^{k-1}}{4}.
    \]
    Also,
    \begin{align*}
        r & \ge k - \frac{2 \log k + \log 16}{\log l} - 1 \ge k - \frac{2 \log k + \log 16}{\log 2} - 1 \ge k - (3 \log k + 5) 
    \end{align*}
    and so, using the fact that $k \ge k_0$, we have
    \begin{align*}
        8 (k-r) (\log(k-r))^{1+\eps} \le 8 (3 \log k + 5) (\log(3 \log k + 5))^{1+\eps} \le (\log k)^{1+\eps}.
    \end{align*}
    It follows by induction that
    \[
        2 k l^r \cdot t(l,k-r) \le 2 k D_\eps (k-r) (\log(k-r))^{1+\eps} \cdot l^{k-1} \le \frac{D_\eps k (\log k)^{1+\eps} \cdot l^{k-1}}{4}.
    \]
    Combining, we obtain
    \begin{align*}
        t(l,k) & \le t(l,k-1) + g(l,k,r-1) + 2 k l^r \cdot t(l,k-r) \\
        & \le \l(\frac{1}{2} + \frac{1}{4} + \frac{1}{4}\r) D_\eps k (\log k)^{1+\eps} \cdot l^{k-1} \\
        &= D_\eps k (\log k)^{1+\eps} \cdot l^{k-1}. \qedhere
    \end{align*}
\end{proof}

\section{Upper bounds for transitive tournaments}
\label{sec:boundstransitive}

In this section we prove \cref{thm:fequality}, that is, we show that $f(l,k) = l^{k-1}$ if $l$ is sufficiently large in terms of $k$. We will follow a similar strategy to that used by Lefmann, R\"{o}dl, and Thomas \cite{LRT1992} when they showed this equality for $l \ge (3 k)^{2 k}$, but our bounds and arguments from the previous section will allow us to improve the dependence on $k$.

The idea of the proof is to find a \defn{strongly $(k-2)$-robust vertex}, that is, a vertex $v$ such that there exists a set $C$ of $k-1$ colours which contains all colours incident to $v$ and which has the property that for all $a\in C$, there is a $(k-2)$-rainbow with colour set $C \sm \{a\}$ that ends at $v$, and also one that starts at $v$. The existence of such a vertex will allow us to show that every $2$-flash in the entire tournament takes its colour from $C$, which will in turn allow us to reduce the problem to the $l = 2$ case, for which Lefmann, R\"{o}dl, and Thomas~\cite{LRT1992} have already shown that the result holds.

\begin{theorem}[Lefmann, R\"{o}dl, Thomas]\label{thm:t2equality}
    For all positive integers $k$, $f(2,k) = 2^{k-1}$.
\end{theorem}

This final ingredient, \cref{thm:t2equality}, is the only part of the proof where we need our tournament to be transitive: our arguments will show that if $t(2,k)=2^{k-1}$ for all $k$, then \cref{thm:fequality} holds with $t(l,k)$ in place of $f(l,k)$. For this reason, we will state and prove most of the results in this section for arbitrary tournaments.

We begin by explaining how the fact that few colours appear in $2$-flashes allows us to reduce the problem to the $l = 2$ case.

\begin{lemma}\label{lem:fewcolflash}
    Let $1 \le m \le l$, let $T = (V, E)$ be an edge-coloured tournament containing no $l$-flash or $k$-rainbow, and assume that for every vertex $v$ the number of colours used by $m$-flashes containing $v$ is at most $c$. Then,
    \[
        \abs{V} \le t(m,k) \l(\frac{l}{m}\r)^c.
    \]
\end{lemma}

We will use a probabilistic argument similar to that used to prove \cref{lem:cinoutcols}. Write $c^{m\text{-flash}}(v)$ for the set of colours $a$ such that there exists an $m$-flash of colour $a$ containing~$v$.

\begin{proof}
    For every colour $a$, choose a subset $L_a \subs \{0, \dots, l-1\}$ of size $m$ uniformly and independently at random. Define
    \[
        U = \l\{ v \in V \st \forall a \in c^{m\text{-flash}}(v),\ l_a(v) \in L_a \r\},
    \]
    and note that $T[U]$ contains no $m$-flash $v_0, \dots, v_m$ in any colour $a$. Indeed, otherwise $l_a(v_0) < \dots < l_a(v_m)$ with $a \in c^{m\text{-flash}}(v_i)$ for every $i$, so $L_a$ has size at least $m+1$, which is a contradiction. Hence, $\abs{U} \le t(m,k)$. On the other hand, for every vertex $v$ we have $\pr(v \in U) = (m/l)^{\abs{c^{m\text{-flash}}(v)}} \ge (m/l)^c$, and so
    \[
        \abs{V} \l(\frac{m}{l}\r)^c \le \sum_{v \in V} \pr(v \in U) = \ex(\abs{U}) \le t(m,k),
    \]
    which implies that
    $
    \abs{V} \le t(m,k) \l({l}/{m}\r)^c$. \qedhere
\end{proof}

Next, we show how the existence of a strongly $(k-2)$-robust vertex implies that few colours occur in 2-flashes.

\begin{lemma}\label{lem:stronglyrobustflash}
    Let $k\geq 2$ and let $T=(V,E)$ be an edge-coloured tournament containing no $k$-rainbow. Suppose that $T$ contains a strongly $(k-2)$-robust vertex. Then at most $k-1$ colours occur in $2$-flashes.
\end{lemma}

\begin{proof}
    Let $u$ be a strongly $(k-2)$-robust vertex in $T$, and let $C$ be the corresponding set of $k-1$ colours. Suppose for a contradiction that there exists a $2$-flash $v_0, v_1, v_2$ of some colour $b \notin C$ in $T$.
    Consider the case that the edge between $u$ and $v_1$ is directed towards $v_1$, and let $a\in C$ be the colour of that edge. Since $u$ is strongly $(k-2)$-robust, some $(k-2)$-rainbow with colour set $C\sm \{a\}$ ends at $u$. Extending this $(k-2)$-rainbow to $v_1$ and $v_2$ yields a $k$-rainbow in $T$, which is a contradiction. The case that the edge between $u$ and $v_1$ is directed towards $u$ is similar.
    Hence, no $2$-flash of a colour not in $C$ can exist in $T$. Since $\abs{C} = k-1$, the result follows.
\end{proof}

The last step will be to show that every sufficiently large tournament with no $l$-flash or $k$-rainbow contains a strongly $(k-2)$-robust vertex. The proof of this result is similar to the proofs of \cref{lem:trecursiveupperbound,lem:grecursiveupperbound}.

\begin{lemma}\label{lem:sizestronglyrobust}
     Let $k \ge 2$ and let $T = (V, E)$ be an edge-coloured tournament containing no $l$-flash and no $k$-rainbow with
     \[
        \abs{V} \ge t(l,k-1) + g(l,k,k-2) + 2 (k-1) (t(l,k-1) + 2 (k-1) l \cdot t(l,k-2)) + 2.
    \]
    Then $T$ contains a strongly $(k-2)$-robust vertex.
\end{lemma}

\begin{proof}
    Define
    \[
        R = \{ v \in V \st \text{no $(k-1)$-rainbow ends at $v$} \}.
    \]
    Then $\abs{R} \le t(l,k-1)$, and every vertex in $V \sm R$ has at most $k-1$ outgoing colours.

    Order the vertices in $V \sm R$ in a way that maximises the number of forward edges of $T[V \sm R]$, and pick a minimal initial segment $P \subs V \sm R$ of this ordering such that $T[P]$ contains a $(k-2)$-in-robust vertex $v$. Then $\abs{P} \le g(l,k,k-2) + 1$, and if $U = \{ u \in V \st v u \in E \}$ denotes the out-neighbourhood of $v$, we have $2 \abs{U} \ge \abs{V \sm (R \cup P)}$ and so
    \[
        2 \abs{U} \ge \abs{V} - \abs{R} - \abs{P} \ge 2 (k-1) (t(l,k-1) + 2 (k-1) l \cdot t(l,k-2)) + 1 
    \]
    which implies that
    \[
        \abs{U} \ge (k-1) (t(l,k-1) + 2 (k-1) l \cdot t(l,k-2)) + 1.
    \]
    For every colour $a \in c^+(v)$, let $U_a = \{ u \in U \st c(v u) = a \}$. Since $v$ has at most $k-1$ outgoing colours, there must exist a colour $a$ such that $\abs{U_a} \ge {\abs{U}}/{(k-1)}$ and so
    \[
        \abs{U_a} \ge t(l,k-1) + 2 (k-1) l \cdot t(l,k-2) + 1.
    \]
    Since $v$ is $(k-2)$-in-robust, we can choose a $(k-2)$-rainbow $W$ without colour $a$ that ends at $v$, and extend it to each of the vertices in $U_a$ by the edges of colour $a$. Let $C$ be the set containing $a$ and the $k-2$ colours appearing in $W$. Then, every outgoing colour of every vertex $u \in U_a$ must be a colour from $C$. In particular, all edges in $T[U_a]$ have one of these $k-1$ colours.

    Define
    \[
        S = \{ u \in U_a \st \text{no $(k-1)$-rainbow in $T[U_a]$ starts at $u$} \}.
    \]
    Then $\abs{S} \le t(l,k-1)$. For every vertex $u \in U_a \sm S$, some $(k-1)$-rainbow in $T[U_a]$ starts at $u$. Since every edge in $T[U_a]$ has a colour from $C$, this $(k-1)$-rainbow uses only colours from $C$, and so every incoming colour of $u$ in $T$ must be a colour from $C$.
    
    Next, for every colour $b \in C$, define
    \begin{align*}
        & U_{a,b} = \{ u \in U_a \st \text{every $(k-2)$-rainbow in $T[U_a]$ ending at $u$ contains colour $b$} \} \text{ and} \\
        & V_{a,b} = \{ u \in U_a \st \text{every $(k-2)$-rainbow in $T[U_a]$ starting at $u$ contains colour $b$} \}.
    \end{align*}
    For every $m \in \{0, \dots, l-1\}$, let
    \[
        U_{a,b}^m = \{ u \in U_{a,b} \st l_b(u) = m \}.
    \]
    If an edge $u w$ in $T[U_{a,b}^m]$ has colour $b$, then $l_b(u) < l_b(w)$ which gives a contradiction. Thus, $T[U_{a,b}^m]$ contains no $(k-2)$-rainbow, so $\abs{U_{a,b}^m} \le t(l,k-2)$ and $\abs{U_{a,b}} \le l \cdot t(l,k-2)$. Similarly, $\abs{V_{a,b}} \le l \cdot t(l,k-2)$.

    Finally, consider $U_a' = U_a \sm (S \cup \bigcup_{b \in C} (U_{a,b} \cup V_{a,b}))$, so
    \[
        \abs{U_a'} \ge \abs{U_a} - \abs{S} - \sum_{b \in C} (\abs{U_{a,b}} + \abs{V_{a,b}}) \ge 1.
    \]
    Let $u \in U_a'$ be arbitrary. For every colour $b \in C$, there is a $(k-2)$-rainbow in $T[U_a]$ without colour $b$ that ends at $u$, and also one that starts at $u$. Since every edge in $T[U_a]$ has a colour from $C$, the colour set of these $(k-2)$-rainbows is $C \sm \set{b}$. This shows that $u$ is strongly $(k-2)$-robust in $T$.
\end{proof}

We can now combine the preceding three lemmas and insert the bounds from the previous section to prove the following.

\begin{corollary}\label{cor:tboundwitht2}
    For every $\eps > 0$, there is $C_\eps$ such that for all $k \ge 2$,
    \[
        t(l,k) \le \max\l\{C_\eps k^3 (\log k)^{1+\eps} \cdot l^{k-2}, t(2,k) \l(\frac{l}{2}\r)^{k-1}\r\}.
    \]
\end{corollary}

\begin{proof}
    Let $D_\eps$ be given by \cref{cor:tgupperbounds}, and let $C_\eps = 12 D_\eps$. Suppose that $T = (V, E)$ is an edge-coloured tournament containing no $l$-flash and no $k$-rainbow with $\abs{V} \ge C_\eps k^3 (\log k)^{1+\eps} \cdot l^{k-2}$. Then, we have
    \begin{align*}
        \abs{V} \ge {} & 12 D_\eps k^3 (\log k)^{1+\eps} \cdot l^{k-2} \\
        \ge {} & 4 D_\eps k^3 (\log k)^{1+\eps} \cdot l^{k-2} + 2 k \cdot D_\eps k (\log k)^{1+\eps} \cdot l^{k-2} + 4 k^2 l \cdot D_\eps k (\log k)^{1+\eps} \cdot l^{k-3} + 2 \\
        \ge {} & g(l,k,k-2) + (2 (k-1) + 1) t(l,k-1) + 4 (k-1)^2 l \cdot t(l,k-2) + 2 \\
        = {} & t(l,k-1) + g(l,k,k-2) + 2 (k-1) (t(l,k-1) + 2 (k-1) l \cdot t(l,k-2)) + 2.
    \end{align*}
    Hence, by \cref{lem:sizestronglyrobust}, $T$ contains a strongly $(k-2)$-robust vertex. Applying \cref{lem:stronglyrobustflash} yields that at most $k-1$ colours appear in $2$-flashes in $T$, and the result now follows from \cref{lem:fewcolflash}.
\end{proof}

In particular, for transitive tournaments, the preceding corollary tells us that $f(l,k) \le \max\{C_\eps k^3 (\log k)^{1+\eps} \cdot l^{k-2}, f(2,k) (l/2)^{k-1}\}$. We can now prove \cref{thm:fequality}.

\begin{proof}[Proof of \cref{thm:fequality}]
    Let $C_\eps$ be given by \cref{cor:tboundwitht2}. The result is trivial for $k=1$ so assume that $k\geq 2$. If $l \ge C_\eps k^3 (\log k)^{1+\eps}$, then $C_\eps k^3 (\log k)^{1+\eps} \cdot l^{k-2} \le l^{k-1}$, and by \cref{thm:t2equality} we know that $f(2,k) (l/2)^{k-1} = l^{k-1}$. Therefore, \cref{cor:tboundwitht2} implies that $f(l,k) \le l^{k-1}$.
\end{proof}

\section{Colourings of arbitrary tournaments}
\label{sec:colouringtournaments}

In this section, we ask how small a tournament can be while still having the property that every colouring of its edges contains an $l$-flash or a $k$-rainbow. For $k = 2$, any non-transitive tournament is an example of this: if the tournament has a cycle, this cycle either contains two consecutive edges with different colours, creating a $2$-rainbow, or the cycle is monochromatic, creating flashes of arbitrary lengths. For $l\geq 2$ and $k \ge 3$, we show that there are tournaments on fewer than $f(l,k)$ vertices with the desired property.

\begin{proof}[Proof of \cref{thm:tworse}]
    Set $n = f(l,k)$ and let $T$ be the ``increasing'' transitive tournament on $[n - 1]$ with the edge from $1$ to $n - 1$ reversed, that is,
    \[
        T = ([n - 1], \{ u v \st 1 \le u < v \le n - 1,\ u v \neq (1, n - 1) \} \cup \{(n - 1, 1) \}).
    \]
    Suppose that $T$ has an edge colouring $c$ with neither $l$-flashes nor $k$-rainbows. Define an edge colouring $c'$ of the ``increasing'' transitive tournament $T'$ on the vertex set $\{0\}\cup [n]$ as follows. For any edge $u v$ of $T$ with $u v \neq (n - 1, 1)$, let $c'(u v) = c(u v)$. Next, for any $v \in [n - 1]$, let $c'(0, v)$ be the colour of any incoming edge of $v$ in $T$, and let $c'(v, n)$ be the colour of any outgoing edge of $v$ in $T$. Finally, let $c'(1, n - 1)$ and $c'(0,n)$ be distinct and entirely new colours.

    Note first that the edge $(0,n)$ is not contained in an $l$-flash or $k$-rainbow since it is in no walk of length greater than $1$.
    Next, the edge $(1, n - 1)$ is not contained in an $l$-flash in $T'$ since there are no other edges with the same colour. It is also not contained in a $k$-rainbow since the only $3$-walk containing that edge is $0, 1, n - 1, n$, but $c'(0,1) = c(n-1,1) = c'(n-1,n)$. On the other hand, for every walk $w$ in $T'$ containing neither $(0,n)$ nor $(1, n-1)$, there is a walk in $T$ with the same colours. Indeed, if $w$ uses the edge $(0, v)$ for $v\in[n-1]$, then by the construction of $c'$ this can be replaced by an edge of $T$ incident to $v$ of the same colour, and the same holds for edges $(v, n)$. Since $T$ contains no $l$-flash and no $k$-rainbow, it follows that $T'$ also contains no $l$-flash and no $k$-rainbow. However, $T'$ is a transitive tournament with $n + 1 = f(l,k) + 1$ vertices, contradicting the definition of $f(l,k)$.
\end{proof}

In the positive direction, we prove that all tournaments with $\calO(l^{k-2} / \sqrt{k})$ vertices can be coloured in a way that avoids $l$-flashes and $k$-rainbows. Our construction is very similar to the construction from the introduction which showed that $f(l,k) \ge l^{k-1}$. However, it only uses those strings whose entries sum to a fixed value. This ensures that no matter how an edge is directed, we can always pick an index where the first string has a smaller value than the second string.

\begin{proof}[Proof of \cref{thm:tlowerbound}]
    Let
    \begin{equation*}
        X = \l\{x \in [l]^{k - 1} \st \sum_{i = 1}^{k-1} x_i = \floor*{\frac{l (k-1)}{2}}\r\}
    \end{equation*}
    which is the largest antichain in the grid poset~\cite{antichain}. A result of Anderson~\cite{Anderson1969} says that there is a constant $C > 0$ (independent of $l$ and $k$) such that $\abs{X} \geq C l^{k - 2}/\sqrt{k}$. Let $T = (V, E)$ be a tournament with at most $C l^{k - 2}/\sqrt{k}$ vertices.
    
    As $\abs{V} \leq \abs{X}$, we can assign to every vertex $v$ a unique $x(v) \in X$. Define an edge colouring of $T$ by picking, for an edge $u v \in E$, a colour $c(u v) \in [k-1]$ such that $x(u)_{c(u v)} < x(v)_{c(u v)}$. By the construction of $X$, this is always possible.

    Since this edge colouring uses at most $k - 1$ colours, there is no $k$-rainbow. If there were an $l$-flash $v_0, \dots, v_l$ of colour $a$, then $x(v_0)_a < x(v_1)_a < \dots < x(v_l)_a$, implying $x(v_l)_a > l$, which is a contradiction. Thus this gives an edge colouring of $T$ without $l$-flashes and $k$-rainbows.
\end{proof}

With additional work, it is possible to slightly strengthen \cref{thm:tlowerbound}. We sketch the argument, which combines our proof with an argument from Buci\'{c}, Letzter, and Sudakov~\cite{BLS2019} and improves the bound to $C l^{k-2} (\log l)^{1/(k-1)} / \sqrt{k}$. The idea is to partition a tournament of size $n$ into $\calO(n / \log l)$ transitive tournaments of size $\calO(\log l)$ each. From the introduction, we know that each of the transitive tournaments can be coloured with $k-1$ colours and no $a$-flash if $a^{k-1} = \Omega(\log l)$. Moreover, the previous proof provides a way to colour the edges between the transitive tournaments with $k-1$ colours and no $b$-flash provided $b^{k-2} / \sqrt{k} = \Omega(n / \log l)$. The tournament will then contain no $(a b)$-flash, so we need $a b \le l$, and this is satisfied if $n = \calO(l^{k-2} (\log l)^{1/(k-1)} / \sqrt{k})$.

\section{Open problems}\label{sec:openproblems}

We have proved that Lefmann, R\"{o}dl, and Thomas's conjecture, \cref{conj:fequality}, holds for $l \ge k^3 (\log k)^{1+o(1)}$, but the conjecture remains open when $l$ is small compared to $k$. The case $l = 3$ (that is, proving $f(3, k) = 3^{k - 1}$) is already of significant interest. In this setting, there are edge-coloured transitive tournaments with more than $l^{k-1}$ vertices that contain no $l$-flash and no $k$-rainbow starting at the first vertex. The existence of $k$-rainbows starting at the first vertex was a crucial ingredient in Lefmann, R\"{o}dl, and Thomas's proof that $f(2, k) = 2^{k - 1}$.

\Cref{conj:tequality} remains wide open for any $l \ge 2$ and $k \ge 3$. Establishing this conjecture for $l = 2$ would be very useful as our results would then imply that the conjecture also holds for $l \geq k^3 (\log k)^{1 + o(1)}$.

In terms of general upper bounds on $f(l,k)$ and $t(l,k)$, an interesting next step would be to reduce the gap between the upper and the lower bounds to a multiplicative factor that is sublinear in $k$. All of our arguments relied on certain vertices having at most $k-1$ outgoing colours, which always added a factor of at least $k$ to our upper bounds.

\begin{problem}
    For positive integers $l$ and $k$, is it true that $f(l,k) = o(k) \cdot l^{k-1}$\textup{?}
\end{problem}

We have also constructed a tournament with $f(l,k)-1$ vertices, every edge colouring of which contains an $l$-flash or a $k$-rainbow. Are there tournaments with $o(f(l,k))$ vertices that satisfy this? We believe random tournaments are good candidates.

\begin{problem}
    What is the minimal $n$ such that, for a uniformly random tournament on $n$ vertices, with high probability every edge colouring contains an $l$-flash or a $k$-rainbow\textup{?}
\end{problem}

Finally, we defined flashes and rainbows in tournaments to be walks, but it is possible to consider the same problem with paths instead. This adds some technical difficulties. For example, when we split a set of vertices according to the longest flash of colour $a$ ending in those vertices, it is no longer guaranteed that each resulting set contains no edge of colour $a$: the vertices in a directed cycle of colour $a$ could all end up in the same set. However, \cref{conj:tequality} could still hold in this more restrictive setting.

\begin{conjecture}
    Let $l$ and $k$ be positive integers and let $T$ be a tournament with $l^{k-1}+1$ vertices. Then every edge colouring of $T$ contains a directed monochromatic path of length $l$ or a directed rainbow path of length $k$.
\end{conjecture}

{
\fontsize{11pt}{12pt}
\selectfont
	
\hypersetup{linkcolor={red!70!black}}
\setlength{\parskip}{2pt plus 0.3ex minus 0.3ex}

}
\end{document}